\documentclass[12pt]{article}
\usepackage[margin=1in]{geometry}
\DeclareMathAlphabet\mathbfcal{OMS}{cmsy}{b}{n}
\usepackage[utf8]{inputenc}
\usepackage[dvipsnames]{xcolor}
\usepackage{enumerate}
\usepackage[utf8]{inputenc}
\usepackage{amsfonts}
\usepackage{amsthm}
\usepackage{amsmath}
\usepackage{lscape}
\usepackage{mathrsfs}
\usepackage{graphicx}
\usepackage{tikz-cd}
\usepackage{calc}
\usepackage{float}
\usepackage{stmaryrd}
\usepackage{multirow}
\usepackage{lipsum}
\usepackage{amssymb}
\usepackage{mathrsfs}
\usepackage{mathtools}
\usepackage{mathdots}
\usepackage{fancyhdr}
\usepackage{tikz}
\usepackage{faktor}
\usepackage{setspace}
\usetikzlibrary{decorations.markings}

\theoremstyle{plain}
\usepackage{sectsty}
\usepackage{kpfonts}

\newtheorem*{proposition*}{Proposition}

\newtheorem*{observation*}{Observation}
\newtheorem*{theorem*}{Theorem}
\newtheorem*{claim*}{Claim}  
\usepackage{blindtext}
\usepackage{fancyhdr}
\title{An extended Kodaira-Spencer functional}
\author{Gabriella Clemente}
\date{}
\allowdisplaybreaks
\makeatletter
\renewcommand\tableofcontents{%
    \@starttoc{toc}%
}
\begin{document}
\maketitle

\begin{abstract}
This note is about an extension of the Kodaira-Spencer functional to Calabi-Yau manifolds of any dimension.
\end{abstract} 

Let $X$ be an $n$-dimensional Calabi-Yau (CY) manifold with complex structure $J,$ and $\Omega$ be a nowhere vanishing holomorphic volume form. Put $\mathcal{A}^{0,q}_{X,J}:=C^{\infty} \big(X,\bigwedge^{0,q}_{{T_X}^*,J}\big),$ and $\mathcal{B}^{p,0}_{X,J}:=C^{\infty} \big(X,\bigwedge^{p,0}_{T_X ,J}\big).$ The tensor product of the graded commutative algebras $\mathcal{A}^{0,\cdot}_{X,J}=\bigoplus_{q=0}^n \mathcal{A}^{0,q}_{X,J}$ and $\mathcal{B}^{\cdot,0}_{X,J}=\bigoplus_{p=0}^n \mathcal{B}^{p,0}_{X,J}$ is the algebra $\mathcal{A}^{0,\cdot}_{X,J} \otimes \mathcal{B}^{\cdot,0}_{X,J}=\bigoplus_{k=0}^{2n} \Big(\bigoplus_{q+p=k} \mathcal{A}^{0,q}_{X,J} \otimes \mathcal{B}^{p,0}_{X,J}\Big)$ of polyvector field valued forms. There is an isomorphism $\vdash \Omega: \mathcal{A}^{0,\cdot}_{X,J} \otimes \mathcal{B}^{\cdot,0}_{X,J} \to \mathcal{A}^{\cdot,\cdot}_{X,J}$ such that for any $p, q,$ \[\vdash \Omega: \mathcal{A}^{0,q}_{X,J} \otimes \mathcal{B}^{p,0}_{X,J} \to \mathcal{A}^{n-p,q}_{X,J},\] where \[\zeta_1 \wedge \dots \wedge\zeta_q \otimes \eta_1 \wedge \dots \wedge\eta_p \vdash \Omega=\iota_{\eta_p} \iota_{\eta_{p-1}} \dots \iota_{\eta_1}\Omega \otimes \zeta_1 \wedge \dots \wedge\zeta_q.\] Let $\partial$ and $\overline{\partial}$ be the Dolbeault operators on $X.$ Define a differential on the shift \[\Big(\mathcal{A}^{0,\cdot}_{X,J} \otimes \mathcal{B}^{\cdot,0}_{X,J} \Big)[1]=\bigoplus_{k=-1}^{2n-1} \Big(\bigoplus_{q+p-1=k} \mathcal{A}^{0,q}_{X,J} \otimes \mathcal{B}^{p,0}_{X,J}\Big)\] by $\overline{\partial}_J (\alpha \otimes \beta)=\overline{\partial} \alpha \otimes \beta,$ and a bracket by $[\alpha \otimes \beta, \alpha' \otimes \beta']=\alpha \wedge \alpha' \otimes [\beta,\beta']^{SN},$ where $[\cdot,\cdot]^{SN}$ is the Schouten-Nijenhuis bracket. The degree of $x \in \mathcal{A}^{0,q}_{X,J} \otimes \mathcal{B}^{p,0}_{X,J}$ is $deg(x)=q+p-1,$ and $(\alpha \otimes \beta) \wedge (\alpha' \otimes \beta')=(-1)^{(deg(\alpha \otimes \beta)+1)(deg(\alpha' \otimes \beta')+1)} (\alpha' \otimes \beta') \wedge (\alpha \otimes \beta).$ The above data of shifted algebra, differential, and bracket defines a DGLA, which will be denoted throughout by $\mathfrak{t},$ and its homogeneous degree $k$ piece by $\mathfrak{t}^k.$ Any form $\gamma \in \mathfrak{t}$ decomposes as a sum $\gamma=\sum_c \gamma_c,$ $\gamma_c \in \mathfrak{t}^c,$ and further $\gamma_c=\sum_{q+p-1=c} \gamma_{p,q},$ where $\gamma_{p,q} \in \mathcal{A}^{0,q}_{X,J} \otimes \mathcal{B}^{p,0}_{X,J}.$ 

On $\mathfrak{t},$ one also has a degree $-1$ differential $\Delta_J:\mathcal{A}^{0,q}_{X,J} \otimes \mathcal{B}^{p,0}_{X,J} \to \mathcal{A}^{0,q}_{X,J} \otimes \mathcal{B}^{p-1,0}_{X,J},$ defined by $\Delta_J \gamma \vdash \Omega=\partial (\gamma \vdash \Omega)$ and satisfying the Tian-Todorov Lemma \[\Delta_J (\gamma_1 \wedge \gamma_2)=(-1)^{deg(\gamma_1)+1} [\gamma_1,\gamma_2]+\Delta_J \gamma_1 \wedge \gamma_2 +(-1)^{deg(\gamma_1)+1} \gamma_1 \wedge \Delta_J \gamma_2.\] Immediately one gets that if $\gamma_1, \gamma_2 \in \ker{\Delta_J},$ $\Delta_J (\gamma_1 \wedge \gamma_2)=(-1)^{deg(\gamma_1)+1} [\gamma_1,\gamma_2].$ 

The holomorphicity of $\Omega$ implies that $\overline{\partial}_J \gamma \vdash \Omega=\overline{\partial} (\gamma \vdash \Omega).$ The isomorphism given by contraction with $\Omega$ will often be denoted by $f,$ and in this notation, the previous remark says that $\overline{\partial}_J=f^{-1} \circ \overline{\partial} \circ f,$ and it is also true that $\Delta_J=f^{-1} \circ \partial \circ f.$ So the operators $\Delta_J$ and $\overline{\partial}_J$ anti-commute because $\partial \circ \overline{\partial}=-\overline{\partial} \circ \partial.$ The inverse of $\Delta_J,$ wherever it is exists, anti-commutes with $\overline{\partial}_J$ as well. 

Next, define a functional on $\mathfrak{t}$ by integration on $X$: \[\gamma \mapsto \int \gamma=\int_X (\gamma \vdash \Omega) \wedge \Omega.\] Certainly, $\int \gamma =\int \gamma_{n, n}$ for dimensional reasons. This integration defines a pairing \[\langle \alpha, \beta\rangle=\int \alpha \wedge \beta\] with respect to which $\overline{\partial}_J$ is graded skew self-adjoint and $\Delta_J$ is graded self-adjoint, meaning that for any homogeneous $\alpha,$ $\langle \overline{\partial}_J \alpha, \beta \rangle=(-1)^{deg(\alpha)+1} \langle \alpha, \overline{\partial}_J \beta \rangle,$ and $\langle \Delta_J \alpha, \beta \rangle=(-1)^{deg(\alpha)}\langle \alpha, \Delta_J \beta \rangle.$ As a result, the inverse of $\Delta_J,$ whenever it is defined, is graded self-adjoint; informally, $[(\Delta_J \mid_{\mathfrak{t}^k})^{-1}]^*=(-1)^k (\Delta_J \mid_{\mathfrak{t}^k})^{-1}$ so that if $\alpha \in dom (\Delta_J \mid_{\mathfrak{t}^k})^{-1} \subset \mathfrak{t}^{k-1},$ $\langle \Delta^{-1}_J \alpha, \beta\rangle=(-1)^k \langle \alpha, \Delta^{-1}_J \beta \rangle.$

\begin{proposition*}
Consider the functional $\Phi:\ker{\Delta_J} \subset \mathfrak{t} \to \mathbb{C}$ with definition \[\Phi (\gamma)=\int-\frac{1}{2} \overline{\partial}_J \Delta^{-1}_J \gamma \wedge \gamma + \frac{1}{6} \gamma \wedge \gamma \wedge \gamma.\] 

\begin{enumerate}
\item The first variation is 
\begin{equation*}
\begin{split}
\frac{d}{dt}\Big|_{t=0} \Phi(\gamma +t \beta)&=\int \beta \wedge \Big[\frac{1}{2} \Delta^{-1}_J \overline{\partial}_J \gamma -\frac{1}{2} \Delta^{-1}_J \overline{\partial}_J \big(\sum_a (-1)^a \gamma_a \big) +\\
&\frac{1}{6} \Big( \gamma \wedge \gamma +\gamma \wedge \big( \sum_c (-1)^{c+1} \gamma_c \big) + \big(\sum_{c,d} (-1)^{c+d} \gamma_c \wedge \gamma_d \big)\Big)\Big].
\end{split}
\end{equation*}

\item The Euler-Lagrange system is \[\big(\frac{1-(-1)^{p+q-1}}{2}\big) \overline{\partial}_J \gamma_{n-p-1,n-q-1}+\big(\frac{2(-1)^{p+q} +1}{6}\big) \sum_v [\gamma_{n-p-v,n-q-v}, \gamma_{v,v}]=0\] for all $0\leq p,q \leq n.$
\end{enumerate}
\end{proposition*}

\begin{proof}
The first term of the functional has the following interpretation thanks to the Hodge decomposition \[\mathcal{A}^{n-p,q}_{X,J}=\partial \big(\mathcal{A}^{n-p-1,q}_{X,J}\big) \oplus \mathcal{H}^{n-p,q}_{\square} (X) \oplus \partial^* \big(\mathcal{A}^{n-p+1,q}_{X,J}\big),\] from which it follows that $\ker{\partial} \cap \mathcal{A}^{n-p,q}_{X,J} = \partial \big(\mathcal{A}^{n-p-1,q}_{X,J}\big) \oplus \mathcal{H}^{n-p,q}_{\square} (X)$ because $\ker{\partial} \cap \partial^* \big(\mathcal{A}^{n-p+1,q}_{X,J}\big)=\{0\}.$ So \[\ker{\Delta_J}=\Delta_J \big(\mathcal{A}^{0,q}_{X,J} \otimes \mathcal{B}^{p+1,0}_{X,J}\big) \oplus f^{-1} \big(\mathcal{H}^{n-p,q}_{\square} (X)\big).\] Let $\gamma \in \ker{\Delta_J},$ which means that $\Delta_J \gamma_{p,q}=0$ for all $0 \leq p,q \leq n.$ Then, $\gamma_{p,q}=\Delta_J \alpha_{p+1,q} +C_{pq}$ for $C_{pq} \in f^{-1} \big(\mathcal{H}^{n-p,q}_{\square} (X)\big),$ and likewise $\gamma_{n-p-1, n-q-1}=\Delta_J \beta_{n-p, n-q-1}+D_{pq}.$ And then,

\begin{equation*}
\begin{split}
\int \overline{\partial}_J \Delta^{-1}_J \gamma \wedge \gamma &=\sum_{p,q} \int \overline{\partial}_J \Delta^{-1}_J \gamma_{p,q} \wedge \gamma_{n-p-1, n-q-1}\\
&=\sum_{p,q} \Big[\int \overline{\partial}_J \Delta^{-1}_J (\Delta_J \alpha_{p+1,q}) \wedge \Delta_J \beta_{n-p,n-q-1}+\int \overline{\partial}_J \Delta^{-1}_J (\Delta_J \alpha_{p+1,q}) \wedge D_{pq}+\\
& \int \overline{\partial}_J \Delta^{-1}_J C_{pq} \wedge \Delta_J \beta_{n-p, n-q-1} +  \int \overline{\partial}_J \Delta^{-1}_J C_{pq} \wedge D_{pq}\Big]\\
&=\sum_{p,q} \Big[\int \overline{\partial}_J \alpha_{p+1,q} \wedge \Delta_J \beta_{n-p,n-q-1}+\int \overline{\partial}_J \alpha_{p+1,q} \wedge D_{pq}+\\
& \int \overline{\partial}_J \Delta^{-1}_J C_{pq} \wedge \Delta_J \beta_{n-p, n-q-1} +  \int \overline{\partial}_J \Delta^{-1}_J C_{pq} \wedge D_{pq}\Big]\\
&=\sum_{p,q} \Big[\int \overline{\partial}_J \alpha_{p+1,q} \wedge \Delta_J \beta_{n-p,n-q-1}+ (-1)^{p+q+1} \int \alpha_{p+1,q} \wedge \overline{\partial}_J D_{pq}\\
&-\int \Delta^{-1}_J \overline{\partial}_J C_{pq} \wedge \Delta_J \beta_{n-p, n-q-1} + (-1)^{p+q+1} \int  \Delta^{-1}_J C_{pq} \wedge \overline{\partial}_J D_{pq}.
\end{split}
\end{equation*} 

But since $D_{pq}$ is the preimage by the isomorphism $f$ of a harmonic form $h_D,$ $\overline{\partial}_J D_{pq}=f^{-1}(\overline{\partial} h_D)=0,$ and of course also $\overline{\partial}_J C_{pq}=0$ for the same reasons. Therefore, \[\int \overline{\partial}_J \Delta^{-1}_J \gamma \wedge \gamma=\sum_{p,q} \int \overline{\partial}_J \alpha_{p+1,q} \wedge \Delta_J \beta_{n-p,n-q-1}.\]

$1.$ The first derivative at $t=0$ is \[\frac{d}{dt}\Big|_{t=0} \Phi(\gamma +t \beta)=\int -\frac{1}{2} \overline{\partial}_J \Delta^{-1}_J \beta \wedge \gamma - \frac{1}{2} \overline{\partial}_J \Delta^{-1}_J \gamma \wedge \beta +\frac{1}{6} \big(\beta \wedge \gamma \wedge \gamma + \gamma \wedge \beta \wedge \gamma +\gamma \wedge \gamma \wedge \beta \big).\] First, note that 

\begin{equation}\label{1eq}
\begin{split}
\int \overline{\partial}_J \Delta^{-1}_J \beta \wedge \gamma &=\int \overline{\partial}_J \Delta^{-1}_J \big( \sum_a \beta_a \big) \wedge \gamma \\
&=\sum_a (-1)^{2a+3} \int \beta_a \wedge \Delta^{-1}_J \overline{\partial}_J \gamma \\
&=-\int \beta \wedge \Delta^{-1}_J \overline{\partial}_J \gamma,
\end{split}
\end{equation}
and also that
\begin{equation*}
\begin{split}
\int \overline{\partial}_J \Delta^{-1}_J \gamma \wedge \beta &= \int \sum_{a, b} \overline{\partial}_J \Delta^{-1}_J \gamma_a \wedge \beta_b \\
&= \int \sum_{a, b} (-1)^{(a+1)(b+1)} \beta_b \wedge \overline{\partial}_J \Delta^{-1}_J \gamma_a.
\end{split}
\end{equation*}
In general, there is nothing wrong in writing $\int \gamma = \int \gamma_{2n-1},$ but the components $\gamma_{p, q}$ of $\gamma_{2n-1}$ do not contribute to the integral unless $p=q=n.$ It will be useful here to think of $\int \gamma,$ $\int \gamma_{2n-1},$ and $\int \gamma_{n, n}$ as interchangeable. Now, since \[deg(\beta_b \wedge \overline{\partial}_J \Delta^{-1}_J \gamma_a)=a+b+3,\]

\begin{equation}\label{2eq}
\begin{split}
\int \overline{\partial}_J \Delta^{-1}_J \gamma \wedge \beta &= \int \sum_{a, b} (-1)^{(a+1)(b+1)} \beta_b \wedge \overline{\partial}_J \Delta^{-1}_J \gamma_a \\
&=\int \sum_a (-1)^{a+1} \beta_{2(n-2)-a} \wedge \overline{\partial}_J \Delta^{-1}_J \gamma_a \\
&=\int \beta \wedge  \overline{\partial}_J \Delta^{-1}_J \big( \sum_a (-1)^{a+1} \gamma_a \big)\\
&=\int \beta \wedge \Delta^{-1}_J \overline{\partial}_J \big( \sum_a (-1)^a \gamma_a \big).
\end{split}
\end{equation}

Since $\gamma_c \wedge \gamma_d \wedge \gamma_e$ is of degree $c+d+e+2,$ 

\begin{equation}\label{3eq}
\begin{split}
\int \gamma \wedge \beta \wedge \gamma &=\int \big(\sum_c \gamma_c \big) \wedge \big(\sum_d \beta_d \big) \wedge \big( \sum_a \gamma_a\big)\\
&=\int \sum_{c,d,e} (-1)^{(c+1)(d+e+2)} \beta_d \wedge \gamma_e \wedge \gamma_c \\
&=\int \sum_{c,e} (-1)^{c+1} \beta_{2n-3-(c+e)} \wedge \gamma_e \wedge \gamma_c\\
&=\int \beta \wedge \gamma \wedge \big(\sum_c (-1)^{c+1} \gamma_c \big).
\end{split}
\end{equation}
A similar computation shows that 
\begin{equation}\label{4eq}
\begin{split}
\int \gamma \wedge \gamma \wedge \beta &=\int \beta \wedge \big(\sum_{c, d} (-1)^{c+d} \gamma_c \wedge \gamma_d \big).
\end{split}
\end{equation}

Therefore, (\ref{1eq}) -- (\ref{4eq}) suggest that the first variation of $\Phi$ with respect to $\gamma$ is the expression
\begin{equation*}
\begin{split}
\frac{d}{dt}\Big|_{t=0} \Phi(\gamma +t \beta)&=\int \beta \wedge \Big[\frac{1}{2} \Delta^{-1}_J \overline{\partial}_J \gamma -\frac{1}{2} \Delta^{-1}_J \overline{\partial}_J \big(\sum_a (-1)^a \gamma_a \big) +\\
&\frac{1}{6} \Big( \gamma \wedge \gamma +\gamma \wedge \big( \sum_c (-1)^{c+1} \gamma_c \big) + \big(\sum_{c,d} (-1)^{c+d} \gamma_c \wedge \gamma_d \big)\Big)\Big].
\end{split}
\end{equation*}

$2.$ The $(n, n)$-part of \[\beta \wedge \Big[\frac{1}{2} \Delta^{-1}_J \overline{\partial}_J \gamma -\frac{1}{2} \Delta^{-1}_J \overline{\partial}_J \big(\sum_a (-1)^a \gamma_a \big) +\frac{1}{6} \Big( \gamma \wedge \gamma +\gamma \wedge \big( \sum_c (-1)^{c+1} \gamma_c \big) + \big(\sum_{c,d} (-1)^{c+d} \gamma_c \wedge \gamma_d \big)\Big)\Big]\] is

\[\sum_{p,q} \beta_{p,q} \wedge \Big[\big(\frac{1-(-1)^{p+q-1}}{2}\big)\Delta^{-1}_J \overline{\partial}_J \gamma_{n-p-1,n-q-1}+\big(\frac{2+(-1)^{p+q}}{6}\big)\sum_v \gamma_{n-p-v,n-q-v} \wedge \gamma_{v,v}\Big]\] and the Euler-Lagrange system of $\Phi$ is found by setting it equal to zero. The Tian-Todorov Lemma implies that the system is of the form
\[\big(\frac{1-(-1)^{p+q-1}}{2}\big) \overline{\partial}_J \gamma_{n-p-1,n-q-1}+\big(\frac{2(-1)^{p+q} +1}{6}\big) \sum_v [\gamma_{n-p-v,n-q-v}, \gamma_{v,v}]=0,\] where $0 \leq p, q \leq n.$
\end{proof}

The critical points of $\Phi$ seem to correspond to generalized deformations of a complex structure in the extended moduli space $H^{*} (X,\Lambda^{*} T_X)[2].$ If $X$ is a CY $3$-fold, then $\Phi$ restricted to $\mathcal{A}^{0,1}_{X,J} \otimes \mathcal{B}^{1,0}_{X,J} \cap \ker{\Delta_J}$ is the Kodaira-Spencer functional from \cite{BCOV}, whose critical points are genuine Maurer-Cartan (MC) elements of $\mathfrak{t},$ describing first order deformations of a complex structure. 

It is possible to extend $\Phi$ to formal power series with values in $\mathfrak{t}.$ For a solution $\hat{\gamma}=\sum_a \hat{\gamma}_a t^a +\Delta_J \alpha (t)$ to the MC equation as specified in Lemma 6.1 \cite{BK}, this extension recovers the integral \[\mathbf{\Phi}=\int -\frac{1}{2} \overline{\partial}_J \alpha \wedge \Delta_J \alpha+\frac{1}{6} \hat{\gamma} \wedge \hat{\gamma} \wedge \hat{\gamma}.\]

\noindent
Gabriella Clemente

\noindent
e-mail: clemente6171@gmail.com

\begin{thebibliography}{9}
\bibitem{BK} S.\ Barannikov and M.\ Kontsevich. Frobenius manifolds and formality of Lie algebras of polyvector fields. \emph{Int.\ Math.\ Res.\ Not.}, Volume 1998, Issue 4, 1998, 201 -- 215.

\bibitem{BCOV} 
M.\ Bershadsky, S.\ Cecotti, H.\ Ooguri and C.\ Vafa. Kodaira-Spencer theory of gravity and exact results for quantum string amplitudes. \emph{Commun.\ Math.\ Phys.}, Volume 165, 1994, 311 -- 427.
\end{thebibliography}
\end{document}